\documentclass[10pt, twosided]{article}
\usepackage{amssymb}
\usepackage{latexsym, amsmath}
\catcode`\@=11
\@addtoreset{equation}{section}

\catcode`\@=12

\numberwithin{equation}{section}
\newtheorem{Theorem}{Theorem}[section]

\newtheorem{Proposition}[Theorem]{Proposition}
\newtheorem{defn}{Definition}[section]

\newcommand{\Rn}{{\mathbb R}^{N}}

\newcommand{\ba} {\beta}
\newcommand{\ga} {\gamma}
\newcommand{\Ga} {\Gamma}
\newcommand{\De} {\Delta}

\newcommand{\na} {\nabla}

\newcommand{\no}{\nonumber}

\newcommand{\R}{\mathbb{R}}

\newcommand{\deb}{\rightharpoonup}
\newcommand{\starstar}{2^{*\!*} }

\newcommand{\smu}{S_{\mu,0}}
\newcommand{\smub}{S_{\mu,\ba}}

\begin{document}

\title{\bf \Large A note on semilinear elliptic equation with biharmonic operator and multiple critical nonlinearities}

\author{\bf \Large Mousomi Bhakta
\footnote
{Department of Science and Technology, University of New England, Armidale, NSW-2350, Australia. 
Email: {mousomi.bhakta@gmail.com}}}

\date{}

\maketitle

\noindent
\begin{abstract}
\small{ We study the existence and non-existence of nontrivial weak solution of  
$$
{\Delta^2u-\mu\frac{u}{|x|^{4}} = \frac{|u|^{q_{\ba}-2}u}{|x|^{\beta}}+|u|^{q-2}u\quad\textrm{in $\Rn$,}}
$$
where $N\geq 5$, $q_{\ba}=\frac{2(N-\ba)}{N-4}$, $0<\ba<4$, $1<q\leq\starstar$ and 
$\mu<\mu_1:=\big(\frac{N(N-4)}{4}\big)^2$. Using   Pohozaev type of identity, we prove the non-existence result when $1<q<\starstar$. On the other hand when the equation has multiple critical nonlinearities i.e. $q=\starstar$ and $-(N-2)^2\leq\mu<\mu_1$, we establish the existence of nontrivial solution using the Mountain-Pass theorem by Ambrosetti and Rabinowitz.}
\end{abstract}

\noindent


\section{Introduction}

In this article  we  study  the singular fourth order elliptic problem:
\begin{equation}
\label{eq:problem}
\begin{cases}
\Delta^2u-\mu\frac{u}{|x|^{4}}=\frac{|u|^{q_{\ba}-2}u}{|x|^{\beta}}+|u|^{q-2}u\quad\text{in}\quad \Rn,\\
u\in D^{2,2}(\Rn).
\end{cases}
\end{equation}
where $\Delta^2u:=\Delta(\Delta u)$ is the biharmonic operator and
\begin{equation}
\label{eq:ass_q_mu}
\begin{cases}
N\geq 5,\quad q_{\ba}=\frac{2(N-\ba)}{N-4}, \quad 0<\ba<4, \quad 1<q\leq\starstar:=\frac{2N}{N-4},\\ \mu<\mu_1=\displaystyle\left[\frac{N(N-4)}{4}\right]^2. 
\end{cases}
\end{equation}

$D^{2,2}(\Rn)$ is the closure of $C^{\infty}_0(\Rn)$ with respect to the norm $\displaystyle\left(\int_{\Rn}|\Delta u|^2 dx\right)^\frac{1}{2}$.

By a weak solution of this above equation we mean there exists $u\in D^{2,2}(\Rn)$, $u\not\equiv 0$ and 

$$\int_{\Rn}\displaystyle\left[\Delta u\Delta\phi-\mu\frac{u\phi}{|x|^{4}}\right]dx=\int_{\Rn}\frac{|u|^{q_{\ba}-2}u\phi}{|x|^{\ba}}dx+\int_{\Rn}|u|^{q-2}u\phi\ dx \quad \forall\ \phi\in C^{\infty}_0(\Rn).$$

As $1<q\leq\starstar$ and $u\in D^{2,2}(\Rn)$, we get that $u\in L^q_{loc}(\Rn)$ and therefore the definition of weak solution makes sense.

It is well known that $\mu_1$ is the
best constant in the  Rellich inequality (See \cite{Rel54}, \cite{Rel69})
\begin{equation}\label{Rellich}
\mu_1\int_{\Rn}|x|^{-4}|u|^2 dx\displaystyle\leq \int_{\Rn}|\Delta u|^2 dx  \quad\textrm{for any}\   \  u\in \mathcal D^{2,2}(\R^N).
\end{equation}
We also recall here the Sobolev inequality:
\begin{equation}\label{Sob}
S^{**}\displaystyle\left(\int_{\Rn}|u|^{\starstar} dx\right)^\frac{2}{\starstar}\displaystyle\leq \int_{\Rn}|\Delta u|^2 dx  \quad\textrm{for any}\   \  u\in \mathcal D^{2,2}(\R^N)
\end{equation}
where $S^{**}>0$ is the Sobolev constant. When $0<\ba<4$, interpolating the  Sobolev equality and the Rellich inequality via H\'{o}lder inequality we obtain,
\begin{equation}\label{CKN}
 C\left(\int_{\Rn}\frac{|u|^{q_{\ba}}}
{|x|^{\beta}}dx\right)^{\frac{2}{q_{\ba}}}\leq \int_{\Rn}|\De u|^{2}dx \quad \textrm{for any $u\in \mathcal D^{2,2}(\Rn)$}
\end{equation}
where $C=C(N,\ba)$ is a positive constant (see \cite{CM}, \cite{Bh2}, \cite{BM}). Note that \eqref{CKN} is the second order version of the celebrated Caffarelli-Kohn-
Nirenberg inequalities \cite{CKN}. As $\mu<\mu_1$, $\displaystyle\left[\int_{\Rn}\left(|\Delta u|^2-\mu\frac{|u|^2}{|x|^4}\right)dx\right]^\frac{1}{2}$ is an equivalent norm in $D^{2,2}(\Rn)$,   since the following inequality holds:
\begin{equation*}
\displaystyle\left(1-\frac{\mu^{+}}{\mu_1}\right)\int_{\Rn}|\Delta u|^2 dx\leq \int_{\Rn}\left(|\Delta u|^2-\mu\frac{|u|^2}{|x|^4}\right)dx\leq\left(1+\frac{\mu^{-}}{\mu_1}\right)\int_{\Rn}|\Delta u|^2 dx
\end{equation*}
where $\mu^{+}=\text{max}(\mu, 0)$ and $\mu^{-}=-\text{min}(\mu, 0)$. We denote this equivalent norm by $||u||$. 

\vspace{3mm}

Existence, nonexistence as well as qualitative properties of nontrivial solutions of elliptic equations with biharmonic operator and with/without singular potentials were recently studied by several authors, but essentially with one critical exponent. We refer \cite{AMM1, AMM2, Bh2, BM, CM, GGS, Lin} and the references there-in. For the  second order elliptic equations, more precisely, equations with the Lapalce/p-Laplace operator, existence of solutions were studied when multiple critical nonlinearities were involved (see \cite{GP, FPR, KL}). In this article we extend the results of \cite{FPR} to the fourth order semilinear equation. The main purpose of this paper is to discuss the existence of solutions to the singular problem
\eqref{eq:problem} when $q=\starstar$ by using variational methods. When $q=\starstar$ in \eqref{eq:problem}, we define the corresponding  energy functional $I$ on $D^{2,2}(\Rn)$ associated with \eqref{eq:problem} as follows:
\begin{equation}\label{I-u}
I(u)=\frac{1}{2}||u||^2-\frac{1}{q_{\ba}}\int_{\Rn}\frac{|u|^{q_{\ba}}}{|x|^{\ba}}dx-\frac{1}{\starstar}\int_{\Rn}|u|^{\starstar} dx\quad \text{for}\ u\in D^{2,2}(\Rn).
\end{equation}
By using Rellich inequality, Sobolev inequality and \eqref{CKN}, it is easy to see that $I$ is a well-defined $C^1$ functional on $D^{2,2}(\Rn)$. The critical points of
$I$ correspond to weak solutions of \eqref{eq:problem} when $q=\starstar$.  The standard method to find the critical points of the functional is via Mountain-Pass Theorem of Ambrosetti and Rabinowitz but note that when $q=\starstar$, equation \eqref{eq:problem} is invariant under the weighted dilation 
\begin{equation}\label{dial}
u(x)\mapsto t^\frac{N-4}{2}u(tx) \quad t>0.
\end{equation}
Therefore it is well known that the mountain pass theorem does not yield critical points, but only the Palais-Smale sequences (see definition \eqref{d:PS}). In this type of situation it is always very important to understand the convergence of the Palais-Smale sequences. As it was already mentioned in \cite{FPR}, we observe here the main difficulty is that there is an asymptotic competition between the energy carried by two critical nonlinearities. If one dominates the other, then there is vanishing of the weakest one and we obtain solution of an equation with only one critical nonlinearity. Therefore the crucial step here is to avoid the dominance of one term on the other.  To overcome this difficulty, in Section 3 we choose the Palais-Smale sequence at "suitable" energy level and doing a careful analysis of concentration, we show that there is a balance between the energies of the two nonlinearities mentioned above, and therefore none can dominate the other. Therefore we could make the full use of conformal invariance of \eqref{eq:problem} under the dialtion \eqref{dial} and we recover the solution of \eqref{eq:problem} when $q=\starstar$. In Section 2, using Pohozaev type of identity, we prove that \eqref{eq:problem} does not have any solution when $q<\starstar$.

\section{Non-existence result when $q<\starstar$}
\begin{Theorem}\label{t:nonex}
Let $\ba$, $q_{\ba}$ and $\mu$ be defined as in \eqref{eq:ass_q_mu}. If $u\in D^{2,2}(\Rn)$ is a weak solution of \eqref{eq:problem} where $1<q<\starstar$, then $u\equiv 0$.
\end{Theorem}

We first prove this theorem under an additional assumption.

\begin{Proposition}\label{p:nonex}
In addition to the assumptions on Theorem \ref{t:nonex}, we assume that $u\in L^q(\Rn)$. Then if $u$ is a weak solution of \eqref{eq:problem} with $1<q< \starstar$,  $u\equiv 0$.
\end{Proposition}

\begin{proof} We prove this proposition by establishing Pohozaev type of identity. A similar result was proved in \cite{Bh2} on bounded domain in $\Rn$. We use the same cut-off function which was used in \cite{Bh2}. More precisely,
for $\epsilon>0$ and $R>0$, we define $\phi_{\epsilon,R}(x)=\phi_{\epsilon}(x)\psi_{R}(x)$ where 
$\phi_{\epsilon}(x)=\phi(\frac{|x|}{\epsilon})$ and $\psi_R(x)=\psi(\frac{|x|}{R})$, $\phi$ and $\psi$ are smooth functions in $\R$ with the properties $0\leq\phi,\psi\leq 1$, with supports of $\phi$ and $\psi$ in $(1,\infty)$ and $(-\infty, 2)$ respectively and $\phi(t)=1$ for $t\geq 2$, and $\psi(t)=1$ for $t\leq 1$.

Let $u\in D^{2,2}(\Rn)$ be a weak solution of \eqref{eq:problem} and $1<q<\starstar$. Then $u$ is smooth away from origin (see \cite[page:\ 235-236]{GGS}) and hence $(x\cdot\na u)\phi_{\epsilon, R}\in C^3_c(\Rn)$. Multiplying equation \eqref{eq:problem} by $(x\cdot\na u)\phi_{\epsilon,R}$ and integrating by parts we obtain

\begin{eqnarray}\label{non1}
\int_{\Rn}\Delta^2 u(x\cdot\na u)\phi_{\epsilon,R}dx
&=&\mu\int_{\Rn}\frac{u(x\cdot\na u)}{|x|^4}\phi_{\epsilon,R}dx+\int_{\Rn}\frac{|u|^{q_{\ba}-2}u}{|x|^{\ba}}(x\cdot\na u)\phi_{\epsilon,R} dx\no\\
&+&\int_{\Rn}|u|^{q-2}u(x\cdot\na u)\phi_{\epsilon,R}dx.
\end{eqnarray}
Proceeding similarly as proved in \cite[Theorem 2.1]{Bh2}, we can show that 
\begin{equation}\label{non2}
\lim_{R\to\infty}\lim_{\epsilon\to 0}\text{RHS}=-\displaystyle\left(\frac{N-4}{2}\right)\left(\mu\int_{\Rn}\frac{|u|^2}{|x|^4}dx+\int_{\Rn}\frac{|u|^{q_{\ba}}}{|x|^{\ba}}dx\right)-\frac{N}{q}\int_{\Rn}|u|^qdx. 
\end{equation}
and
\begin{equation}\label{non3}
\lim_{R\to\infty}\lim_{\epsilon\to 0}\text{LHS}=-\displaystyle\left(\frac{N-4}{2}\right)\int_{\Rn}|\Delta u|^2dx.
\end{equation}
Therefore substituting back \eqref{non2} and \eqref{non3} in \eqref{non1} we obtain
\begin{equation}\label{non4}
-\displaystyle\left(\frac{N-4}{2}\right)\left(\int_{\Rn}|\Delta u|^2dx-\mu\int_{\Rn}\frac{|u|^2}{|x|^4}dx-\int_{\Rn}\frac{|u|^{q_{\ba}}}{|x|^{\ba}}dx\right)=-\frac{N}{q}\int_{\Rn}|u|^qdx.
\end{equation}
Also from the equation \eqref{eq:problem} we have
$$\int_{Rn}|\Delta u|^2dx=\mu\int_{\Rn}\frac{|u|^2}{|x|^4}dx+\int_{\Rn}\frac{|u|^{q_{\ba}}}{|x|^{\ba}}dx+\int_{\Rn}|u|^qdx.$$
Comparing this with \eqref{non4} we obtain
\begin{equation}\label{non5}
\left(\frac{N-4}{2}-\frac{N}{q}\right)\int_{\Rn}|u|^qdx=0.
\end{equation}
As $q<\starstar$, \eqref{non5} implies $u\equiv 0$.
\end{proof}

\vspace{3mm}

{\bf Proof of Theorem \ref{t:nonex}:} By virtue of Proposition \ref{p:nonex}, proof of this theorem follows once we prove $u\in L^q(\Rn)$. To prove this we choose a cut-off function $\phi_{\epsilon,R}\in C^{\infty}_0(\Rn\setminus \{0\})$ as in the proof of Proposition \ref{p:nonex}. Then by choosing $\phi_{\epsilon,R}u$ as a test function we obtain,
\begin{equation}\label{non6}
\int_{\Rn}\Delta u\Delta(\phi_{\epsilon,R}u)dx=\mu\int_{\Rn}\frac{\phi_{\epsilon,R}|u|^2}{|x|^4}dx+\int_{\Rn}\frac{\phi_{\epsilon,R}|u|^{q_{\ba}}}{|x|^\ba}dx+\int_{\Rn}\phi_{\epsilon,R}|u|^qdx.
\end{equation}
Therefore,
\begin{equation*}
\text{LHS}=\int_{\Rn}|\Delta u|^2\phi_{\epsilon,R}dx+\int_{\Rn}u\Delta u\Delta\phi_{\epsilon,R}dx+2\int_{\Rn}\Delta u\na u\cdot\na\phi_{\epsilon,R}dx.
\end{equation*}
Hence from \eqref{non6} we obtain
\begin{eqnarray}\label{non7}
\int_{\Rn}\phi_{\epsilon,R}|u|^qdx &\leq& |\mu|\int_{\Rn}\frac{|u|^2}{|x|^4}dx+\int_{\Rn}\frac{|u|^{q_{\ba}}}{|x|^{\ba}}dx+\int_{\Rn}|\Delta u|^2dx+\int_{\Rn}|u||\Delta u||\Delta\phi_{\epsilon,R}|dx\no\\
&+& 2\int_{\Rn}|\Delta u||\na u||\na\phi_{\epsilon,R}|dx.
\end{eqnarray}
We denote the last two integrals in the RHS by $I_1$ and $I_2$ respectively. Now our aim is to show that $I_1$ and $I_2$ are uniformly bounded by a constant independent of $\epsilon$ and $R$. To see this,
\begin{eqnarray*}
I_1&=&\int_{\Rn}|u||\Delta u||\Delta\phi_{\epsilon,R}|dx\\
&=&\int_{\Rn}|u||\Delta u||\psi_R\Delta\phi_{\epsilon}+2\na\psi_R\na\phi_{\epsilon}+\phi_{\epsilon}\Delta\psi_R|dx\\
&\leq&\displaystyle\int_{\epsilon\leq|x|\leq2\epsilon}|u||\Delta u|\left(\frac{c}{|\epsilon|^2}+\frac{c}{\epsilon|x|}\right)dx+\displaystyle\int_{\{\epsilon\leq|x|\leq2\epsilon\}\cap\{R\leq|x|\leq2R\}}|u||\Delta u|\frac{c}{\epsilon R}dx\\
&+&\displaystyle\int_{R\leq|x|\leq2R}|u||\Delta u|\left(\frac{c}{R^2}+\frac{c}{R|x|}\right)dx
\end{eqnarray*}
Note that in the first integral $\frac{1}{\epsilon}\leq\frac{2}{|x|}$, in the second integral $\frac{1}{\epsilon R}\leq\frac{4}{|x|^2}$ and in the third integral $\frac{1}{R}\leq\frac{2}{|x|}$. Therefore we get,
 
$$I_1\leq C\int_{\Rn}\frac{|u|}{|x|^2}|\Delta u|dx\leq C\displaystyle\left(\int_{\Rn}\frac{|u|
^2}{|x|^4}dx\right)^\frac{1}{2}\left(\int_{\Rn}|\Delta u|^2dx\right)^\frac{1}{2}\leq C'.$$

\begin{eqnarray*}
I_2&=&\int_{\Rn}|\Delta u||\na u||\na\phi_{\epsilon,R}|dx\\
&\leq& c\int_{\epsilon\leq|x|\leq2\epsilon}\frac{|\Delta u||\na u|}{\epsilon}dx+c\int_{R\leq|x|\leq2R}\frac{|\Delta u||\na u|}{R}dx\\
&\leq& C\int_{\Rn}|\Delta u|\frac{|\na u|}{|x|}dx\leq C\displaystyle\left(\int_{\Rn}\frac{|\na u|
^2}{|x|^2}dx\right)^\frac{1}{2}\left(\int_{\Rn}|\Delta u|^2dx\right)^\frac{1}{2}\leq C''
\end{eqnarray*}
where for the last inequality we have used the Hardy's inequality (\cite[(2.1)]{Bh2}) on $\na u$ since $\na u\in D^{1,2}(\Rn)$.

Hence from \eqref{non7} we obtain,
$\int_{\Rn}\phi_{\epsilon,R}|u|^q dx\leq C$, where $C$ is a positive constant independent of $\epsilon$ and $R$. Therefore letting $\epsilon\to 0$ and then $R\to\infty$, we obtain $u\in L^q(\Rn)$. Hence the theorem follows. 
\hfill{$\square$}

\section{Existence Result  when $q=\starstar$}

In this section we consider the equation 
\begin{equation}\label{ex1}
\Delta^2u-\mu\frac{u}{|x|^{4}}=\frac{|u|^{q_{\ba}-2}u}{|x|^{\beta}}+|u|^{\starstar-2}u\quad\text{in}\quad \Rn; \quad u\in D^{2,2}(\Rn).
\end{equation}

We define,
\begin{equation}\label{smu}
S_{\mu,0}=\inf_{u\in D^{2,2}(\Rn),u\not= 0}\frac{\displaystyle\int_{\Rn}\left(|\Delta u|^2-\mu\frac{|u|^2}{|x|^4}\right)dx}{\displaystyle\int_{\Rn}\left(|u|^{\starstar}dx\right)^\frac{2}{\starstar}}
\end{equation}
\begin{equation}\label{smub}
S_{\mu,\ba}=\inf_{u\in D^{2,2}(\Rn),u\not= 0}\frac{\displaystyle\int_{\Rn}\left(|\Delta u|^2-\mu\frac{|u|^2}{|x|^4}\right)dx}{\displaystyle\left(\int_{\Rn}\frac{|u|^{q_{\ba}}}{|x|^{\ba}}dx\right)^\frac{2}{q_{\ba}}}.
\end{equation}
We prove the next theorem in the spirit of \cite{FPR}.
\begin{Theorem}\label{t:ex}
Let $-(N-2)^2\leq\mu<\mu_1$ and $\ba$, $q_{\ba}$ be defined as in \eqref{eq:ass_q_mu}. Then there exists at least one non-trivial weak solution of \eqref{ex1} which belongs to $D^{2,2}(\Rn)\cap C^4(\Rn\setminus\{0\})$.
\end{Theorem}
\begin{defn}\label{d:PS}
We say $\{u_n\}\subset D^{2,2}(\Rn)$ is a Palais-Smale sequence (in short, PS sequence) of $I$ at  level $c$ if  
$$\lim_{n\to\infty}I(u_n)=c \quad\text{and}\quad I'(u_n)\to 0 \quad\text{in}\ \big(D^{2,2}(\Rn)\big)'$$ where $\big(D^{2,2}(\Rn)\big)'$ is the dual space of $D^{2,2}(\Rn)$.
\end{defn}

Here we recall the following version of the Mountain-Pass theorem by Ambrosetti and Rabinowitz (see \cite{AR})
\begin{Theorem}\label{t:MP}
Let $(V,||.||)$ be a Banach space and $F\in C^1(V)$. we assume that\\
(i)$F(0)=0$,\\
(ii) There exists $\alpha>0$
and $\epsilon>0$ such that $F(u)\geq\alpha$ whenever $||u||=\epsilon$,\\
(iii) There exists $u\in V$ such that $\limsup_{t\to\infty}F(tu)<0$.\\
Let $t_u>0$ be such that $||t_uu||>\epsilon$ and $F(t_uu)<0$ and let
$$c_u:=\inf_{\ga\in\Ga}\sup_{t\in[0,1]}F(\ga(t)),$$
 $\text{where}\quad \Ga:=\{\ga\in C^0([0,1],V):\ga(0)=0 \ \text{and}\  \ga(1)=t_uu\}.$ Then there exists a PS sequence of $F$ at level $c_u$.
\end{Theorem}

\subsection{Case 1: $\mu\geq 0$}
\begin{Proposition}\label{p:ex1}
There exists a PS sequence of $I$ at a level $c$ where 
\begin{equation}\label{ex2}
0<c<c^{*}:=\text{min}\displaystyle\left\{\frac{2}{N}S_{\mu,0}^\frac{N}{4}, \ \frac{1}{2}\left(\frac{4-\ba}{N-\ba}\right)S_{\mu,
\ba}^\frac{N-\ba}{4-\ba}\right\}.
\end{equation}
\end{Proposition}
\begin{proof}
{\bf Step 1:} First we will prove that $I$, as defined in \eqref{I-u}, satisfy all the conditions in Theorem \ref{t:MP}. To see this,\\
(i)$I(0)=0$.\\
(ii) Using the definition of $\smu$ and $\smub$ we obtain
\begin{eqnarray*}
I(u)&\geq&  \frac{1}{2}||u||^2-\frac{\smub^{-\frac{q_{\ba}}{2}}}{q_{\ba}}||u||^{q_{\ba}}-\frac{\smu^{-\frac{\starstar}{2}}}{\starstar}||u||^{\starstar}\\
&=&||u||^2\displaystyle\left[\frac{1}{2}-\frac{\smub^{-\frac{q_{\ba}}{2}}}{q_{\ba}}||u||^{q_{\ba}-2}-\frac{\smu^{-\frac{\starstar}{2}}}{\starstar}||u||^{\starstar-2}\right].
\end{eqnarray*}
As $q_{\ba}>2$, we can choose $\epsilon>0$ small enough such that if $||u||=\epsilon$, terms in the bracket of the above expression is strictly positive and therefore, $I(u)\geq \alpha>0$ when $||u||=\epsilon$. \\
(iii)Given $u\in D^{2,2}(\Rn)$ such that $u\not=0$, it is easy to see from the definition of $I$ that $\lim_{t\to\infty}I(tu)=-\infty$. Then we choose $t_u>0$ corresponding to $u$ such that $I(tu)<0$ for all $t>t_u$ and $||t_uu||>\epsilon$. We define,
$$\Ga_{u}:=\{\ga\in C^0([0,1],D^{2,2}(\Rn)):\ga(0)=0,  \ \ga(1)=t_uu\}$$ and
$$c_{u}:=\inf_{\ga\in\Ga_{u}}\sup_{t\in[0,1]}I(\ga(t)).$$
Therefore by applying Theorem \ref{t:MP}, we obtain a PS sequence of $I$ at a level $c_{u}$. Also by the definition of $c_{u}$, we have $c_{u}\geq\alpha>0$.\\

{\bf Step 2:} We aim to show that there exists $u\in D^{2,2}(\Rn)$ such that $u\not\equiv 0$ and 
\begin{equation}\label{ex3}
0<c_u<c^{*}
\end{equation}
where $c_u$ is as defined in step 1.\\
To prove this, let $u\in D^{2,2}(\Rn)$ be the non-negative extremal of $\smu$. Existence of such $u$ was proved in Theorem 1.3 and Theorem 5.1 in \cite{BM} for the case $\mu>0$ and in \cite{Swan} for the case $\mu=0$. Corresponding to that $u$, we define $t_u$ and $c_u$ as in step 1, which yields,
\begin{equation}\label{ex4}
0<c_u\leq\sup_{t\geq 0}I(tu)\leq\sup_{t\geq 0}f(t)
\end{equation}
where $f(t):=\frac{t^2}{2}||u||^2-\frac{t^{\starstar}}{\starstar}\int_{\Rn}|u|^{\starstar}dx$.
As $u$ is the extremal for $\smu$, by standard method it can be shown that upto a multiplicative constant $u$ is a non-negative weak solution of  
\begin{equation}\label{ex5}
\Delta^2 v-\mu\frac{v}{|x|^4}=v^{\starstar-1} \quad\text{in}\ \Rn.
\end{equation}
If $\theta>0$ is the constant for which $\theta u$ is a solution of \eqref{ex5}, then $||u||^2=\theta^{\starstar-2}\int_{\Rn}|u|^{\starstar}dx$. Therefore by \eqref{ex4} and the definition of $f(t)$ there, we obtain
\begin{eqnarray*}
0<c_u&\leq& \sup_{t\geq 0}\displaystyle\left(\frac{t^2}{2}\theta^{\starstar-2}-\frac{t^{\starstar}}{\starstar}\right)\int_{\Rn}|u|^{\starstar}dx\\
&=&\displaystyle\left(\frac{1}{2}-\frac{1}{\starstar}\right)\theta^{\starstar}\int_{\Rn}|u|^{\starstar}dx=\frac{2}{N}\theta^{\starstar}\int_{\Rn}|u|^{\starstar}dx=\frac{2}{N}\smu^{\frac{N}{4}}.
\end{eqnarray*}
If the equality would hold in the above inequality i.e. $c_u=\frac{2}{N}\smu^{\frac{N}{4}}$, then 
$$0<c_u=\sup_{t\geq 0}I(tu)=\sup_{t\geq 0}f(t).$$ Let $t_1$ and $t_2$ be the two points where the two supremum are attained respectively. Then we get, 
$$f(t_1)-\frac{t_1^{q_{\ba}}}{q_{\ba}}\int_{\Rn}\frac{|u|^{q_{\ba}}}{|x|^{\ba}}dx=f(t_2),$$ which in turn implies, $f(t_2)<f(t_1)$ and this is a contradiction to the fact that $t_2$ is the supremum of $f$. Therefore we obtain $0<c_u<\frac{2}{N}\smu^{\frac{N}{4}}.$

Now note that if $\frac{2}{N}\smu^{\frac{N}{4}}\leq\frac{1}{2}\left(\frac{4-\ba}{N-\ba}\right)S_{\mu,
\ba}^\frac{N-\ba}{4-\ba}$, then we are done. Otherwise we choose $u\in D^{2,2}(\Rn)\setminus\{0\}$ which is a non-negative extremal of $\smub$ (which exists by Theorem 1.3 and Theorem 5.1 in \cite{BM}). Now we Proceed as before and here we  replace $f$ in \eqref{ex4} by $g$ where $$g(t):=\frac{t^2}{2}||u||^2-\frac{t^{q_{\ba}}}{q_{\ba}}\int_{\Rn}\frac{|u|^{q_{\ba}}}{|x|^{\ba}}dx$$ which gives now the contradiction $g(t_1)-\frac{t^{\starstar}}{\starstar}\displaystyle\int_{\Rn}|u|^{\starstar}dx=g(t_2)$ as before. Hence the claim in Step 2 follows.
\end{proof}
\begin{Proposition}\label{p:ex2}
Let $\{u_n\}\subset D^{2,2}(\Rn)$ be a PS sequence of $I$ at a level $c\in(0, c^{*})$ and $u_n\deb 0$ in $D^{2,2}(\Rn)$. Then there exists $\epsilon=\epsilon(N, \mu, c, \ba)>0$ such that 
$$\text{either}\ \limsup_{n\to\infty}\int_{B_r(0)}|u_n|^{\starstar}dx=0 \quad\text{or}\quad \limsup_{n\to\infty}\int_{B_r(0)}|u_n|^{\starstar}dx\geq\epsilon \quad\forall\ r>0. $$
\end{Proposition}
\begin{proof}
{\bf Step 1:} Let $D$ be an arbitrary compact set in $\Rn\setminus\{0\}$. Then we claim, upto a subsequence
\begin{equation}\label{ex6}
\lim_{n\to\infty}\int_{D}|\Delta u_n|^2dx=\lim_{n\to\infty}\int_{D}\frac{|u_n|^2}{|x|^4}dx=\lim_{n\to\infty}\int_{D}\frac{|u_n|^{q_{\ba}}}{|x|^{\ba}}dx=\lim_{n\to\infty}\int_{D}|u_n|^{\starstar}dx=0.
\end{equation} 
To see this, note that $u_n\deb 0$ in $D^{2,2}(\Rn)$ implies that $u_n\longrightarrow 0$ in $L^q_{loc}(\Rn)$ for $q\in[2,\starstar)$. Since $\ba>0$, we have $q_{\ba}<\starstar$ and therefore
\begin{equation}\label{ex7}
\lim_{n\to\infty}\int_{D}\frac{|u_n|^2}{|x|^4}dx=\lim_{n\to\infty}\int_{D}\frac{|u_n|^{q_{\ba}}}{|x|^{\ba}}dx=0.
\end{equation}
Concerning the other two inequality, first we note that as $u_n\deb 0$ in $D^{2,2}(\Rn)$,  we obtain $\{u_n\}$ is a bounded sequence in $D^{2,2}(\Rn)$ and therefore,
$$I(u_n)-\frac{1}{2}(I'(u_n),u_n)= c+o(1)||u_n||= c+o(1),$$
which in turn implies, 
\begin{equation}\label{ex7'} 
 \displaystyle\left(\frac{1}{2}-\frac{1}{q_{\ba}}\right)\int_{\Rn}\frac{|u_n|^{q_{\ba}}}{|x|^{\ba}}dx+\left(\frac{1}{2}-\frac{1}{\starstar}\right)\int_{\Rn}|u_n|^{\starstar}dx= c+o(1) \quad\text{as}\ n\to\infty.
 \end{equation} 
Therefore, \begin{equation}\label{ex8}
\int_{\Rn} |u_n|^{\starstar}dx\leq \frac{cN}{2}+o(1) \quad\text{and}\ \int_{\Rn}\frac{|u_n|^{q_{\ba}}}{|x|^{\ba}}dx\leq 2c\left(\frac{N-\ba}{4-\ba}\right)+o(1) \quad\text{as}\ n\to\infty.
\end{equation}
Let $\phi\in C^{\infty}_0(\Rn\setminus \{0\})$ such that $\phi\equiv 1$ in $D$. By using Holder inequality, Rellich's compactness theorem and Sobolev inequality it's easy to see that
\begin{equation}\label{ex9}
\int_{\Rn}\Delta u_n\Delta(\phi^2u_n)dx=\int_{\Rn}|\Delta(\phi u_n)|^2dx+o(1).
\end{equation}
Therefore by taking $\phi^2u_n$ as a test function in the equation \eqref{ex1}  and using \eqref{ex7}  and the uniform boundness of $u_n$ we obtain,
\begin{eqnarray*}
o(1)=(I'(u_n),\phi^2u_n) &=&\int_{\Rn}\Delta u_n\Delta(\phi^2u_n)dx-\int_{\Rn}\phi^2|u_n|^{\starstar}dx+o(1)\\
&=&\int_{\Rn}|\Delta(\phi u_n)|^2dx-\int_{\Rn}(\phi u_n)^2|u_n|^{\starstar-2}dx+o(1)\\
&\geq& ||\phi u_n||^2-\int_{\Rn}(\phi u_n)^2|u_n|^{\starstar-2}dx+o(1)
\end{eqnarray*}
Therefore  by using Holder inequality and the definition of $\smu$ we reduce,
\begin{equation}\label{ex10}
 ||\phi u_n||^2\leq \smu^{-1}||\phi u_n||^2\displaystyle\left(\int_{\Rn}|u_n|^{\starstar}dx\right)^\frac{\starstar-2}{\starstar}+o(1).
\end{equation}
Now using \eqref{ex8} in \eqref{ex10}, we obtain
$$||\phi u_n||^2\displaystyle\left[1-\smu^{-1}\left(\frac{cN}{2}\right)^\frac{4}{N}\right]\leq o(1).$$
Here we note that as $c<c^{*}$, the term inside the bracket, in the LHS of above expression, is strictly positive. This in turn implies $\lim_{n\to\infty}||\phi u_n||^2=0$. Therefore $\lim_{n\to\infty}\int_{D}|\Delta u_n|^2dx=0$ and from this we conclude by Sobolev inequality,  $\lim_{n\to\infty}\int_{D}|u_n|^{\starstar}dx=0$ . This completes step 1.\\

For $r>0$, we define
$$I_1=\limsup_{n\to\infty}\int_{B_r(0)}|u_n|^{\starstar}dx, \quad I_2=\limsup_{n\to\infty}\int_{B_r(0)}\frac{|u_n|^{q_{\ba}}}{|x|^{\ba}}dx$$ and $$ I_3=\limsup_{n\to\infty}\int_{B_r(0)}\displaystyle\left(|\Delta u_n|^2-\mu\frac{|u_n|^2}{|x|^4}\right) dx.$$
By step 1, the above three quantities are well defined and independent of the choice of  $r>0$. 

\vspace{2mm}

{\bf Step 2:} In this step we complete the proof of this Proposition. 
Let $\phi\in C^{\infty}_0(\Rn)$ such that $\phi\equiv 1$ in $B_r(0)$. Then $\displaystyle\left(\int_{\Rn}|\phi u_n|^{\starstar}dx\right)^\frac{2}{\starstar}\leq\smu^{-1}||\phi u_n||^2$. Application of  step 1 in this expression yields,
$$\displaystyle\left(\int_{B_r(0)}|u_n|^{\starstar}dx\right)^\frac{2}{\starstar}\leq\smu^{-1}\int_{B_r(0)}\left(|\Delta u_n|^2-\mu\frac{|u_n|^2}{|x|^4}\right)dx+o(1),$$ which implies $I_1^\frac{2}{\starstar}\leq\smu^{-1}I_3$. Similarly we can also prove $I_2^\frac{2}{q_{\ba}}\leq\smub^{-1}I_3$. On the other hand as $\lim_{n\to\infty}(I'(u_n),\phi u_n)=0$, using step 1 and the definition of $I_1$, $I_2$ and $I_3$ we  obtain $I_3\leq I_1+I_2$. Therefore, $I_1^\frac{2}{\starstar}\leq\smu^{-1}I_1+\smu^{-1}I_2$. Therefore, 
$$I_1^\frac{2}{\starstar}\displaystyle\left(1-\smu^{-1}I_1^\frac{4}{N}\right)\leq\smu^{-1}I_2.$$ 
We note that, from \eqref{ex8} we have $I_1\leq\frac{cN}{2}$. Therefore we obtain $$I_1^\frac{2}{\starstar}\displaystyle\left(1-\smu^{-1}\left(\frac{cN}{2}\right)^\frac{4}{N}\right)\leq\smu^{-1}I_2.$$ 
Therefore since $c<c^{*}<\frac{2}{N}\smu^\frac{N}{4}$, we obtain $I_1^\frac{2}{\starstar}\leq C_1I_2$ for some constant $C_1=C_1(N,\mu, c)>0$. Similarly we can prove that $I_2^\frac{2}{q_{\ba}}\leq C_2I_1$ for some $C_2=C_2(N,\mu,c,\ba)>0$. Combining these two inequalities we obtain either $I_1=I_2=0$ or there exists  $\epsilon=\epsilon(N,\mu, \ba, c)>0$ such that $\{I_1\geq\epsilon\quad\text{and}\quad I_2\geq\epsilon\}$.
\end{proof}

\vspace{3mm}

{\bf Proof of Theorem \ref{t:ex} in the case $\mu\geq 0$:} Let $\{u_n\}$ be a PS sequence of $I$ at level $c\in(0, c^{*})$. We claim that $\limsup_{n\to\infty}\int_{\Rn}|u_n|^{\starstar}dx>0$. 
We prove this claim by method of contradiction. Therefore we assume,
\begin{equation}\label{ex11}
\lim_{n\to\infty}\int_{\Rn}|u_n|^{\starstar}dx=0
\end{equation}
Using \eqref{ex11}, it is easy to check that $\{u_n\}$ is bounded. Indeed we have, 
\begin{equation*}
c+o(1)||u_n||\geq I(u_n)-\frac{1}{q_{\ba}}(I'(u_n), u_n)=\displaystyle\left(\frac{1}{2}-\frac{1}{q_{\ba}}\right)||u_n||^2+o(1)
\end{equation*}
 and hence the boundedness follows. Now using \eqref{ex11} , we estimate $(I'(u_n), u_n)$ and obtain $$o(1)||u_n||=||u_n||^2-\int_{\Rn}\frac{|u_n|^{q_{\ba}}}{|x|^{\ba}}dx+o(1).$$ As $u_n$ is bounded, from the above expression we obtain
\begin{equation*}
||u_n||^2 =\displaystyle\int_{\Rn}\frac{|u_n|^{q_{\ba}}}{|x|^{\ba}}dx+o(1),
\end{equation*}
which by \eqref{smub} yields
\begin{eqnarray}\label{ex12}
&\displaystyle\left(\int_{\Rn}\frac{|u_n|^{q_{\ba}}}{|x|^{\ba}}dx\right)^\frac{2}{q_{\ba}}\smub \leq  \int_{\Rn}\frac{|u_n|^{q_{\ba}}}{|x|^{\ba}}dx+o(1),\no\\
&\displaystyle\left(\int_{\Rn}\frac{|u_n|^{q_{\ba}}}{|x|^{\ba}}dx\right)^\frac{2}{q_{\ba}}\left[\smub-\left(\int_{\Rn}\frac{|u_n|^{q_{\ba}}}{|x|^{\ba}}dx\right)^\frac{4-\ba}{N-\ba}\right] \leq o(1).
\end{eqnarray}
As in \eqref{ex8}, we can prove that $\displaystyle\int_{\Rn}\frac{|u_n|^{q_{\ba}}}{|x|^{\ba}}dx\leq 2c\frac{N-\ba}{4-\ba}+o(1)$ as $n\to\infty$. Plugging this inequality in \eqref{ex12} and using the upper bound of $c$, we obtain
\begin{equation}\label{un-qb}
\lim_{n\to\infty}\int_{\Rn}\frac{|u_n|^{q_{\ba}}}{|x|^{\ba}}dx=0.
\end{equation}
Therefore \eqref{un-qb}, along with \eqref{ex11} is a contradiction to \eqref{ex7'} as $c>0$ and thus the claim follows i.e. 
$$\limsup_{n\to\infty}\int_{\Rn}|u_n|^{\starstar}dx>0.$$
We define $\limsup_{n\to\infty}\int_{\Rn}|u_n|^{\starstar}dx=d$, which is positive by the claim above. Since $\{u_n\}$ is bounded, upto a subsequence $u_n\deb u$ for some $u\in D^{2,2}(\Rn)$. If $u\not= 0$, we are done as $u$ will be the nontrivial weak solution of \eqref{ex1}.  Therefore we may assume that $u_n\deb 0$ in $D^{2,2}(\Rn)$. Here we
set, $\delta=\text{min}(d, \frac{\epsilon}{2})$, where $\epsilon>0$ is the same which we obtain from Proposition \ref{p:ex2}. Define
$$Q_n(r)=\int_{B_{r}(0)}|u_n|^{\starstar}dx.$$
Therefore for any $\delta'\in(0,\delta)$, there exists a sequence $r_n\in\R^{+}$ such that upto a subsequence $Q_n(r_n)=\delta'.$ We define, $v_n(x)=r_n^\frac{N-4}{2}u_n(r_nx)$. Then $v_n\in D^{2,2}(\Rn)$ and satisfies 
\begin{equation}\label{ex13}
\int_{B_1(0)}|v_n|^{\starstar}dx=\delta'.
\end{equation} 
It is easy to check that $\{v_n\}$ is a PS sequence of $I$ at level $c$.  By the scaling invariance of the norm in $D^{1,2}(\Rn)$ and the boundedness of the sequence $\{u_n\}$, it follows  that $\{v_n\}$ is bounded in $D^{2,2}(\Rn)$. Therefore,  we may assume that there exists $v_0\in D^{2,2}(\Rn)$ such that, upto a subsequence $v_n\deb v_0$ in $D^{2,2}(\Rn)$.\\

Claim: $v_0\not=0$.\\

Suppose the claim is not true. Therefore, $\{v_n\}$ satisfies the properties of Proposition \ref{p:ex2} with the same $\epsilon$ mentioned there. Therefore we have 
\begin{equation}\label{ex14}
\text{either }\quad\lim_{n\to\infty}\int_{B_1(0)}|v_n|^{\starstar}dx=0 \quad\text{or}\quad \limsup_{n\to\infty}\int_{B_1(0)}|v_n|^{\starstar}dx\geq\epsilon.
\end{equation}
Since in \eqref{ex13}, $0<\delta'<\frac{\epsilon}{2}$, \eqref{ex14} is a contradiction to \eqref{ex13} and therefore $v_0\not=0$.\\
Also note that as $\{v_n\}$ is a PS sequence for $I$, we have 
\begin{equation}\label{ex15}
\displaystyle\int_{\Rn}\left[\Delta v_n\Delta\phi-\mu\frac{v_n\phi}{|x|^4}\right]dx=\int_{\Rn}\frac{|v_n|^{q_{\ba}-2}v_n\phi}{|x|^{\ba}}dx+\int_{\Rn}|v_n|^{\starstar-2}v_n\phi dx+o(1) \quad\forall\ \phi\in D^{2,2}(\Rn).
\end{equation}
Using Vitaly's convergence theorem, we pass to the limit in \eqref{ex15} and we obtain $v_0$ is a nontrivial weak solution of \eqref{ex1}. Therefore we can write $\Delta^2 v_0=g(x,v_0)v_0$ where 
$$g(x,v_0)=\frac{\mu}{|x|^4}+\frac{|v_0|^{q_{\ba}-2}}{|x|^{\ba}}+|v_0|^{\starstar-2}.$$ If $D$ is an arbitrary compact subset of $\Rn\setminus\{0\}$, then there exist two constants $C_1(D), C_2(D)>0$ such that
$|g(x,v_0)|\leq C_1(D)|v_0|^{\starstar-2}+C_2(D)$ for every $x\in D$. Therefore it follows from \cite[Lemma B.3]{VVV} that $v_0\in C^3(\bar D)$.
As $D$ is arbitrary compact set in $\Rn\setminus\{0\}$, we obtain $u\in C^3(\Rn\setminus\{0\})$.  Using the $C^3$ regularity outside the origin and the fact that nonlinear part in \eqref{ex1} depends superlinearly on $v_0$, it follows that $\Delta^2 v_0$ is locally Lipschitz continuous.  Therefore using Schauder estimates \cite[Theorem 2.19]{GGS}, we obtain $v_0\in C^{4,\alpha}_{loc}(\Rn\setminus\{0\})$ for any $\alpha\in(0,1)$. This completes the proof.
\hfill$\square$

\subsection{Case 2: $\mu\in[-(N-2)^2, 0)$}
When $\mu<0$, $S_{\mu, 0}$ is not achieved (see theorem 1.3 in \cite{BM}). Therefore we give here an alternate proof to recover the full range $\mu\in[-(N-2)^2, 0)$.\\
We define,
\begin{equation}\label{smur}
\smu^{rad}=\inf_{\scriptstyle u\in \mathcal D^{2,2}(\Rn)
\atop\scriptstyle u=u(|x|)~,~u\ne 0}\frac{\displaystyle
\int_{\Rn}\left(|\De u|^{2}-\mu\frac{|u|^{2}}{|x|^{4}}\right)dx}
{\displaystyle\left(\int_{\Rn}|u|^{\starstar}dx\right)^{2/\starstar}}
\end{equation}
and for  $0<\ba<4$,  
\begin{equation}\label{smubr}
\smub^{rad}=\inf_{\scriptstyle u\in \mathcal D^{2,2}(\Rn)
\atop\scriptstyle u=u(|x|)~,~u\ne 0}\frac{\displaystyle
\int_{\Rn}\left(|\De u|^{2}-\mu\frac{|u|^{2}}{|x|^{4}}\right)dx}
{\displaystyle\left(\int_{\Rn}\frac{|u|^{q_{\ba}}}{|x|^{\beta}}dx\right)^{2/q_{\ba}}}
\end{equation}
It is being known that,  $\smu^{rad}$ and $\smub^{rad}$ are always achieved (see \cite[Theorem 1.1]{BM}). In that article the authors also proved  that the corresponding extremals are positive (see \cite[Theorem 1.2]{BM}). It also follows from \cite{BM} and \cite{Swan} that, when $\mu\geq 0$ the following equality holds: $\smub^{rad}=\smub$, where $0\leq\ba<4$, which is not necessarily be true if $\mu<0$. It follows from \cite[Theorem 5.2]{BM} that $\smub<\smub^{rad}$, when $\mu<<0$. For the second order elliptic operator, Catrina and Wang have proved in their celebrated paper \cite{CW} that, for any $\mu<0$, there exists $\tilde\ba_{\mu}\in (0,2)$ such  that for every $\tilde\ba\in(0,\tilde\ba_{\mu})$, no minimizer of 
\begin{equation}\label{tl-smub}
S_{\mu, \tilde\ba}:=\inf_{\scriptstyle u\in \mathcal D^{1,2}(\Rn)
\atop\scriptstyle u\ne 0}\frac{\displaystyle
\int_{\Rn}\left(|\nabla u|^{2}-\mu\frac{|u|^{2}}{|x|^{2}}\right)dx}
{\displaystyle\left(\int_{\Rn}\frac{|u|^{q_{\tilde\ba}}}{|x|^{\tilde\beta}}dx\right)^{2/q_{\tilde\ba}}}
\end{equation}
is radially symmetric, where $q_{\tilde\ba}=\frac{2(N-\tilde\ba)}{N-2}$. 

Since we have existence of nonnegative extremal for $\smub^{rad}$ when $0\leq\ba<4$,  we can carry out the proofs of Proposition \ref{p:ex1} and \ref{p:ex2} by restricting to radial functions and by replacing $\smu$ and $\smub$ in the definition \eqref{ex2} of $c^{*}$ by $\smu^{rad}$ and $\smub^{rad}$ respectively. This proves Theorem \ref{t:ex} in the case $\mu\in[-(N-2)^2, 0)$.

\vspace{3mm}
 
 \small
\noindent
{\bf Acknowledgment.} 
The author acknowledges the valuable suggestions of the referee, which helped in great extent to improve the manuscript. The author also acknowledges the support of Australian Research Council (ARC).

\label{References}

\end{document}